\documentclass[11pt]{article}

\usepackage[utf8]{inputenc}
\usepackage{etex}
\usepackage[all]{xy}
\usepackage{cite}
\usepackage{amsfonts}
\usepackage{amsthm}
\usepackage{enumerate}
\usepackage{graphicx}
\usepackage{mathrsfs}
\usepackage{bm}
\usepackage{amssymb,amsmath} 
\usepackage{makecell}
\usepackage{tikz}
\usepackage{pgf}
\usepackage{tikz}
\usetikzlibrary{patterns}
\usepackage{pgffor}
\usepackage{pgfcalendar}
\usepackage{pgfpages}
\usepackage{shuffle,yfonts}
\usepackage{mathtools}
 \allowdisplaybreaks
\DeclareFontFamily{U}{shuffle}{}
\DeclareFontShape{U}{shuffle}{m}{n}{ <-8>shuffle7 <8->shuffle10}{}

\newcommand{\barz}{{\bar\zeta}}

\newcommand{\A}{{\rm A}}

\newcommand{\bfk}{{\boldsymbol{\sl{k}}}}

\newcommand{\bfx}{{\boldsymbol{\sl{x}}}}
\newcommand{\bfz}{{\boldsymbol{\sl{z}}}}

 \allowdisplaybreaks
\usetikzlibrary{arrows,shapes,chains}
 \allowdisplaybreaks

\catcode`!=11
\let\!int\int \def\int{\displaystyle\!int}
\let\!lim\lim \def\lim{\displaystyle\!lim}
\let\!sum\sum \def\sum{\displaystyle\!sum}
\let\!sup\sup \def\sup{\displaystyle\!sup}
\let\!inf\inf \def\inf{\displaystyle\!inf}
\let\!cap\cap \def\cap{\displaystyle\!cap}
\let\!max\max \def\max{\displaystyle\!max}
\let\!min\min \def\min{\displaystyle\!min}
\let\!frac\frac \def\frac{\displaystyle\!frac}
\catcode`!=12

\let\oldsection\section
\renewcommand\section{\setcounter{equation}{0}\oldsection}

\allowdisplaybreaks

\DeclareMathOperator*{\dep}{dep}
\DeclareMathOperator{\Li}{Li}

\def\N{\mathbb{N}}

\theoremstyle{plain}
\newtheorem{thm}{Theorem}[section]

\newtheorem{cor}[thm]{Corollary}

\theoremstyle{definition}
\newtheorem{qu}{Question}[section]
\newtheorem{defn}{Definition}[section]

\begin{document}
\title{\bf A Unified Approach to Formal Arakawa-Kaneko Type Zeta Values via Iterated Integrals}
\author{
{Ende Pan${}^{a,}$\thanks{Email: 32082@qzc.edu.cn},\quad Ce Xu${}^{b,}$\thanks{Email: cexu2020@ahnu.edu.cn} \quad and\quad Jianqiang Zhao${}^{c,}$\thanks{Email: zhaoj@ihes.fr}}\\[1mm]
\small  a. College of Teacher Education, Quzhou University, \\ \small Quzhou 324022, P.R. China\\
\small  b. School of Mathematics and Statistics, Anhui Normal University, \\ \small Wuhu 241002, P.R. China\\
\small  c. Department of Mathematics, The Bishop's School, La Jolla, CA 92037, USA.
}

\date{}
\maketitle

\noindent{\bf Abstract.}
We introduce a formal multiple polylogarithm function (FMPF) and general formal Arakawa-Kaneko zeta values (GFAKZVs), and establish several useful properties analogous to those of the ordinary Arakawa-Kaneko zeta values by employing the method of iterated integrals, a powerful tool originating in algebraic topology. In particular, we prove various duality and sum formulas involving FMPFs and GFAKZVs. This allows us to derive, in a uniform manner, a number of known results concerning Arakawa-Kaneko zeta values and Kaneko-Tsumura $\psi$-values that were previously obtained by different approaches, thereby highlighting the unifying power of the iterated integral formalism.


\noindent{\bf MSC code (Mathematics Subject Classification):} 11M32, 11M99.

\section{Introduction}
Many objects in mathematics have been studied through their associated zeta functions, the most famous of which is undoubtedly the Riemann zeta function, whose study dates back to Euler and has been central to number theory for centuries. Over the past few decades, numerous generalizations have been discovered and extensively investigated. These include, in particular, multiple zeta functions-which generalize the Riemann zeta to higher dimensions and were independently introduced by Hoffman and Zagier in the early 1990s-and the Arakawa-Kaneko zeta function, which arose from the study of poly-Bernoulli numbers and their connections to multiple zeta values. The study of the special values of these functions, often expressed as infinite series, has proved particularly fruitful and is deeply connected to various other branches of mathematics and even theoretical physics (see, e.g., \cite{Brown2012,Broadhurst1996}). To explore the structures of such infinite series, researchers have employed a variety of methods, including algebraic geometry~\cite{Brown2012,Deligne2010}, complex analysis and special functions~\cite{XuZhao2024b,XuZhao2022May}, differential equations~\cite{Li2008,Zagier2012}, and others. In this paper, we utilise the theory of iterated integrals, first developed by Chen in the context of path spaces and fundamental groups~\cite{KTChen1971,KTChen1977}, to give a unified treatment of various results related to Arakawa-Kaneko zeta values.

We begin with some basic notation. Let $\N$ denote the set of positive integers.
A finite sequence $\bfk:=(k_1,\ldots, k_r)\in\N^r$ is called a \emph{composition}. We put
\begin{equation*}
 |\bfk|:=k_1+\cdots+k_r,\quad \dep(\bfk):=r,
\end{equation*}
and call them the \emph{weight} and the \emph{depth} of $\bfk$, respectively. If $k_r>1$, then $\bfk$ is called \emph{admissible}. As a convention, we write $\{m\}_r$ for the sequence of $m$ repeated $r$ times.

\begin{defn}\label{defn-fmplf}
Let $z_j\equiv z_{j}(t)\ (j=1,2,\ldots,r)$ be meromorphic functions of $t$ with at most logarithmic singularities. For any $\bfz:=(z_1,\ldots,z_r)$ and composition $\bfk:=(k_1,\ldots,k_r)$, we define the \emph{formal multiple polylogarithm function} (FMPF) by
\begin{align}\label{equ-defn-fmplf}
&L_{z_0}\left(\!\! {\begin{array}{*{20}{c}}
   {\bfz}  \\
   {\bfk}  \\
\end{array}} ;x\!\right)\equiv L_{z_0}\left(\!\! {\begin{array}{*{20}{c}}
   {{z_1},{z_2}, \ldots ,{z_r}}  \\
   {k_1,k_2,\ldots,k_r}  \\
\end{array}} ;x\!\right)\nonumber\\&:=\int_0^x z_{1}(t)dt \big(z_{0}(t)dt \big)^{k_1-1}\cdots z_{r}(t)dt \big(z_{0}(t)dt \big)^{k_r-1}\nonumber\\& =\int_0^x z_{1}dt \big(z_{0}dt \big)^{k_1-1}\cdots z_{r}dt \big(z_{0}dt \big)^{k_r-1},
\end{align}
where
\[\int_{0}^z f_1(t)dt\, f_2(t)dt\cdots f_k(t)dt:=\int\limits_{0<t_1<\cdots<t_k<z}f_1(t_1)f_2(t_2)\cdots f_k(t_k)\, dt_1dt_2\cdots dt_k
\]
is the iterated integral first studied by K.T. Chen in the 1960s~\cite{KTChen1971,KTChen1977}.
\end{defn}

Clearly, we have
\begin{align}\label{diff-RF}
\frac{d}{dx} L_{z_0}\left(\!\! {\begin{array}{*{20}{c}}
   {{z_1},{z_2}, \ldots ,{z_r}}  \\
   {k_1,k_2,\ldots,k_r}  \\
\end{array}} ;x\!\right)= \left\{ {\begin{array}{*{20}{c}}  L_{z_0}\left(\!\! {\begin{array}{*{20}{c}}
   {{z_1}, \ldots ,z_{r-1},{z_r}}  \\
   {k_1,\ldots,k_{r-1},k_r-1}  \\
\end{array}} ;x\!\right)z_0(x)
   {,\ \ k_r>1,}  \\
   { L_{z_0}\left(\!\! {\begin{array}{*{20}{c}}
   {{z_1}, \ldots ,{z_{r-1}}}  \\
   {k_1,\ldots,k_{r-1}}  \\
\end{array}} ;x\!\right)z_r(x),\;\;\;k_r= 1.}  \\
\end{array} } \right.
\end{align}

For example, if we take $z_{j}(t)=\frac{1}{1-t}$ for all $j=1,\ldots,r$ and $z_0(t)=\frac1{t}$ in \eqref{defn-fmplf}, then $L_{z_0}$ reduces to the classical \emph{multiple polylogarithm function}
\begin{align*}
 L_{z_0}\left(\!\! {\begin{array}{*{20}{c}}
   {{z_1},{z_2}, \ldots ,{z_r}}  \\
   {k_1,k_2,\ldots,k_r}  \\
\end{array}} ;x\!\right)&=\Li_{k_1,\ldots,k_r}(x)
:= \sum\limits_{1 \le {n_1} <  \cdots  < {n_r}} {\frac{{{x^{{n_r}}}}}{{n_1^{{k_1}}n_2^{{k_2}} \cdots n_r^{{k_r}}}}},\quad x \in [- 1,1).
\end{align*}
More generally, for any $\bfk=(k_1,\dotsc,k_r)\in\N^r$ and $\bfx=(x_1,\ldots,x_r)\in \mathbb{C}^r$, the classical $r$-variable \emph{multiple polylogarithm} is defined by
\begin{align*}
\Li_{\bfk}(\bfx)\equiv \Li_{k_1,\dotsc,k_r}(x_1,\dotsc,x_r):=\sum_{0<n_1<\cdots<n_r} \frac{x_1^{n_1}\dotsm x_r^{n_r}}{n_1^{k_1}\dotsm n_r^{k_r}},
\end{align*}
which converges if $|x_j\cdots x_r|<1$ for all $j=1,\dotsc,r$. It can be analytically continued to a multi-valued meromorphic function on $\mathbb{C}^r$ (see \cite{Zhao2007d}).
Clearly, if $k_r>1$, then $\Li_{k_1,\ldots,k_r}(1)=\zeta(k_1,\ldots,k_r)$. The classical \emph{multiple zeta values} (MZVs) are defined by (see~\cite{H1992,DZ1994})
\begin{align*}
\zeta(k_1,\dotsc,k_r):=\sum\limits_{0<m_1<\dotsb<m_r } \frac{1}{m_1^{k_1}\dotsm m_r^{k_r}},
\end{align*}
for positive integers $k_1,\dotsc,k_r$ with $k_r>1$. The systematic study of MZVs began in the early 1990s with the works of Hoffman~\cite{H1992} and Zagier~\cite{DZ1994}. Owing to their surprising and sometimes mysterious appearances in many areas of mathematics and theoretical physics, these special values have attracted considerable attention over the past three decades (see, for example, the book~\cite{Zhao2016}).

If we set $z_{j}(t)=\frac{2}{1-t^2}$ for all $j=1,\ldots,r$ and $z_0(t)=\frac1{t}$ in \eqref{defn-fmplf}, then we obtain the \emph{Kaneko-Tsumura multiple $\A$-function}~\cite{KanekoTs2018b,KT2018,KanekoTs2019}
\begin{align*}
L_{z_0}\left(\!\! {\begin{array}{*{20}{c}}
   {{z_1},{z_2}, \ldots ,{z_r}}  \\
   {k_1,k_2,\ldots,k_r}  \\
\end{array}} ;x\!\right)=\A(k_1,\ldots,k_r;x)
:= 2^r\sum\limits_{1 \le {n_1} <  \cdots  < {n_r}\atop n_i\equiv i\ {\rm mod}\ 2} {\frac{{{x^{{n_r}}}}}{{n_1^{{k_1}}n_2^{{k_2}} \cdots n_r^{{k_r}}}}},\quad x\in \left[ { - 1,1} \!\right).
\end{align*}
In particular, Kaneko and Tsumura~\cite{KanekoTs2018b,KanekoTs2019} introduced and studied a new kind of level-two multiple zeta values, called \emph{multiple $T$-values} (MTVs), defined by
\begin{align*}
T(k_1,k_2,\ldots,k_r):&=2^r \sum_{0<m_1<\cdots<m_r\atop m_i\equiv i\ {\rm mod}\ 2} \frac{1}{m_1^{k_1}m_2^{k_2}\cdots m_r^{k_r}}\nonumber\\
&=2^r\sum\limits_{0<n_1<\cdots<n_r} \frac{1}{(2n_1-1)^{k_1}(2n_2-2)^{k_2}\cdots (2n_r-r)^{k_r}}.
\end{align*}
Clearly, if $k_r>1$, then $\A(k_1,\ldots,k_r;x)=T(k_1,k_2,\ldots,k_r)$. There are, of course, many other generalizations of multiple zeta values, such as multiple $t$-values and $q$-multiple zeta values (see \cite{KanekoTs2024,KomatsuLuca2025}).

\begin{defn}\label{defn-GPAKZVs}(General formal Arakawa-Kaneko zeta values, GFAKZVs for short)
Let $z_j\equiv z_{j}(t)\ (j=1,2,\ldots,r)$ be meromorphic functions of $t$ with at most logarithmic singularities. For any $\bfz:=(z_1,\ldots,z_r)$, composition $\bfk:=(k_1,\ldots,k_r)$ and $p\in \N_0$, define
\begin{align}\label{defn-equ-GPAKZVs}
\xi^{(x)}_{z_0}\left(\!\! {\begin{array}{*{20}{c}}
   {{z_1},{z_2}, \ldots ,{z_r}}  \\
   {k_1,k_2,\ldots,k_r}  \\
\end{array}} ;{\begin{array}{*{20}{c}}
   {\{z_{r+1}\}_p}  \\
  \{1\}_p \\
\end{array}}\!\!\right):=\int_0^x L_{z_0}\left(\!\! {\begin{array}{*{20}{c}}
   {\{z_{r+1}\}_p}  \\
   {\{1\}_p}  \\
\end{array}} ;t\!\right) z_0(t)L_{z_0}\left(\!\! {\begin{array}{*{20}{c}}
   {{z_1},\ldots ,{z_r}}  \\
   {k_1,\ldots,k_r}  \\
\end{array}} ;t\!\right)dt.
\end{align}
\end{defn}

For example, if all $z_{j}(t)=\frac{1}{1-t}\ (j=1,2,\ldots,r+1)$ and $z_0(t)=\frac1{t}$ with $x\rightarrow 1$ in \eqref{defn-equ-GPAKZVs}, then it reduces to the \emph{Arakawa-Kaneko zeta values}~\cite{AM1999,X2021}
\begin{align*}
&\xi(k_1,k_2\ldots,k_r;p+1)=\frac{(-1)^{p}}{p!}\int\limits_{0}^1 \frac{\log^{p}(1-x){\mathrm{Li}}_{{{k_1},{k_2}, \cdots ,{k_r}}}\left(\! x \!\right)}{x}dx,
\end{align*}
while if all $z_{j}(t)=\frac{2}{1-t^2}\ (j=1,2,\ldots,r+1)$ and $z_0(t)=\frac1{t}$ with $x\rightarrow 1$, then it becomes the \emph{Kaneko-Tsumura $\psi$-values}~\cite{KanekoTs2018b,KanekoTs2019}
\begin{align*}
&\psi(k_1,k_2\ldots,k_r;p+1)=\frac{(-1)^{p}}{p!}\int\limits_{0}^1 \frac{\log^{p}\left(\!\frac{1-x}{1+x}\!\right)\A(k_1,k_2,\ldots,k_r;x)}{x}dx,
\end{align*}
and if all $z_{j}(t)=\frac{2}{1+t^2}\ (j=1,2,\ldots,r)$, $z_{r+1}=\frac{2}{1-t^2}$ and $z_0(t)=\frac1{t}$ with $x\rightarrow 1$ in \eqref{defn-GPAKZVs}, then we obtain the \emph{Kaneko-Tsumura $\bar{\psi}$-values}~\cite[Eq. (3.28)]{XuZhao2020b}
\begin{align*}
&\bar{\psi}(k_1,k_2\ldots,k_r;p+1)=\frac{(-1)^{p}}{p!}\int\limits_{0}^1 \frac{\log^{p}\left(\!\frac{1-x}{1+x}\!\right)B(k_1,k_2,\ldots,k_r;x)}{x}dx,
\end{align*}
where for $x \in \left[ { - 1,1} \right]$,
\begin{align*}
{\rm B}(k_1,\dotsc,k_r;x)&: =
\left\{
\begin{array}{ll} (-1)^m \mathrm{i} {\rm A}(k_1,\dotsc,k_{2m-1};\mathrm{i}x),
   &\quad \hbox{if $r=2m-1$};  \\
   (-1)^m{\rm A}(k_1,\dotsc,k_{2m};\mathrm{i}x), &\quad\hbox{if $r=2m$}.   \\
\end{array}
\right.\\
&=\int_0^x \frac{2dt}{1+t^2}\left(\!\frac{dt}{t}\!\right)^{k_1-1}
\cdots\frac{2dt}{1+t^2}\left(\!\frac{dt}{t}\!\right)^{k_r-1}\\
&=2^r\sum_{0<n_1<n_2<\cdots<n_r} \frac{(-1)^{n_r-r}x^{2n_r-r}}{(2n_1-1)^{k_1}(2n_2-2)^{k_2}\cdots (2n_r-r)^{k_r}}.
\end{align*}

For any composition $\bfk$ and $\Re(s)>0$, the \emph{Arakawa-Kaneko zeta function}~\cite{AM1999}, the \emph{Kaneko-Tsumura $\psi$-function}~\cite{KanekoTs2018b} and the \emph{Kaneko-Tsumura $\bar\psi$-function}~\cite{XuZhao2020b} are defined respectively by
\begin{align*}
\xi(k_1,k_2\ldots,k_r;s):=&\, \frac{1}{\Gamma(s)} \int\limits_{0}^\infty \frac{t^{s-1}}{e^t-1}\Li_{k_1,k_2,\ldots,k_r}(1-e^{-t})dt,\\
\psi(k_1,k_2\ldots,k_r;s)=&\, \frac{1}{\Gamma(s)}\int_0^\infty \frac{t^{s-1}}{\sinh(t)} {\rm A}(k_1,k_2\ldots,k_r;\tanh t/2)dt,\\
\bar\psi(k_1,k_2\ldots,k_r;s):=&\,\frac{1}{\Gamma(s)} \int\limits_{0}^\infty \frac{t^{s-1}}{\sinh(t)}{\rm B}({k_1,\dotsc,k_r};\tanh(t/2))\, dt.
\end{align*}
In addition to the works cited above, many results on Arakawa-Kaneko zeta functions and Kaneko-Tsumura $\psi$-functions can be found in~\cite{Hamahata2011,Chen2019,Coppo2010,Kuba2010,PX2019,XYZ2022,Young2014,ZhengYang2023}, while the following papers are mainly concerned with their special values~\cite{Hoshi2023,Ito2021,KanekoTs2024,LuoSi2023,Nishibiro2024,Xu2019,Yamamoto2022}. For example, several generalizations of these functions are given in \cite{Chen2019,PX2019}. Furthermore, in~\cite{AM1999,KT2018}, the special values $\xi(k_1,k_2\ldots,k_r;s)$ at positive integers are analytically computed and expressed in terms of multiple zeta values. Xu and Zhao~\cite[Thm. 4.6]{XuZhao2020a} proved that all Kaneko-Tsumura $\psi$-values can be expressed in terms of multiple $T$-values. Recently, Kawasaki \cite{Kawasaki2025} provided an explicit formula for these relations.

In this paper, we establish several explicit relations involving the formal multiple polylogarithm function and the general formal Arakawa-Kaneko zeta values by means of iterated integrals. This approach enables us to derive, in a unified way, a number of known results that were previously proved by various methods.

\section{Duality Formulas for Formal Arakawa-Kaneko Zeta Values}

For a composition $\bfk_r=(k_1,k_2,\ldots,k_r)\in \N^r$ with $|\bfk_r|:=k_1+k_2+\cdots+k_r$, we adopt the following notation:
\begin{align*}
&\overrightarrow{\bfk_j}:=(k_1,k_2,\ldots ,k_j),\quad \overleftarrow{\bfk_j}:=(k_r,k_{r-1},\ldots,k_{r+1-j}),\\
&|\overrightarrow{\bfk_j}|:=k_1+k_2+\cdots+k_j,\quad |\overleftarrow{\bfk_j}|:=k_r+k_{r-1}+\cdots+k_{r+1-j},\\
&\overrightarrow{\bfk}_r^{-}:=(k_1,\ldots,k_{r-1},k_r-1),\quad \overleftarrow{\bfk}_r^{-}:=(k_r,\ldots,k_{2},k_1-1),
\end{align*}
with $\overrightarrow{\bfk_0}=\overleftarrow{\bfk_0}:=\emptyset$ and $|\overrightarrow{\bfk_0}|=|\overleftarrow{\bfk_0}|:=0$. For $\bfz_r=(z_1,z_2,\ldots,z_r)$, we set
\begin{align*}
&\overrightarrow{\bfz_j}:=(z_1,z_2,\ldots ,z_j)\quad \text{and}\quad \overleftarrow{\bfz_j}:=(z_r,z_{r-1},\ldots,z_{r+1-j}).
\end{align*}
Building on our previous work, we now establish a more general duality formula, which allows us to derive numerous duality relations for both alternating multiple zeta values and alternating multiple $T$-values.

\begin{thm}\label{} For positive integers $p,q,r$ and a composition ${\bfk}=(k_1,\ldots,k_r)$ with $k_1,k_2,\ldots,k_r\in \N\setminus\{1\}$,
\begin{align}\label{GKTSAKF-DUALT-NOEXP}
&\xi^{(x)}_{z_0}\left(\!\! {\begin{array}{*{20}{c}}
   {\{z_1\}_{q-1},\bfz_r}  \\
   {\{1\}_{q-1},\overrightarrow{\bfk}_r^{-}}  \\
\end{array}} ;{\begin{array}{*{20}{c}}
   {\{z_{r+1}\}_p}  \\
  \{1\}_p  \\
\end{array}}\!\!\right)
-(-1)^{|{\bfk}|}\xi^{(x)}_{z_0}\left(\!\! {\begin{array}{*{20}{c}}
   {\{z_{r+1}\}_{p},z_r,\ldots,z_2}  \\
   {\{1\}_{p-1},\overleftarrow{\bfk}_r^{-}}  \\
\end{array}} ;{\begin{array}{*{20}{c}}
   {\{z_{1}\}_q}  \\
  \{1\}_q\\
\end{array}}\!\!\right)\nonumber\\
&=\sum\limits_{j=0}^{r-1} (-1)^{\mid\stackrel{\leftarrow}{{\bfk}}_j\mid}\sum\limits_{i=1}^{k_{r-j}-2}(-1)^{i-1} L_{z_0}\left(\!\! {\begin{array}{*{20}{c}}
   {\{z_{r+1}\}_p,\overleftarrow{\bfz_j}}  \\
   {\{1\}_{p-1},{\overleftarrow{\bfk}_j},i+1}  \\
\end{array}} ;x\!\right)L_{z_0}\left(\!\! {\begin{array}{*{20}{c}}
   {\{z_1\}_{q-1},{\overrightarrow{\bfz}}_{r-j}}  \\
   {\{1\}_{q-1},\overrightarrow{\bfk}_{r-j-1},k_{r-j}-i}  \\
\end{array}} ;x\!\right)   \nonumber\\
&\quad+\sum\limits_{j=0}^{r-2}(-1)^{\mid\stackrel{\leftarrow}{{\bfk}}_{j+1}\mid} \left\{\begin{array}{l} L_{z_0}\left(\!\! {\begin{array}{*{20}{c}}
   {\{z_{r+1}\}_p,{\overleftarrow{\bfz}}_j}  \\
   {\{1\}_{p-1},{\overleftarrow{\bfk}_{j+1}}}  \\
\end{array}} ;x\!\right) L_{z_0}\left(\!\! {\begin{array}{*{20}{c}}
   {\{z_1\}_{q-1},{\overrightarrow{\bfz}}_{r-j}}  \\
   {\{1\}_{q-1},{\overrightarrow{\bfk}}_{r-j-1},1}  \\
\end{array}} ;x\!\right)   \\
\quad\quad-L_{z_0}\left(\!\! {\begin{array}{*{20}{c}}
   {\{z_{r+1}\}_p,{\overleftarrow{\bfz}}_{j+1}}  \\
   {\{1\}_{p-1},{\overleftarrow{\bfk}_{j+1}},1}  \\
\end{array}} ;x\!\right) L_{z_0}\left(\!\! {\begin{array}{*{20}{c}}
   {\{z_1\}_{q-1},{\overrightarrow{\bfz}}_{r-j-1}}  \\
   {\{1\}_{q-1},{\overrightarrow{\bfk}}_{r-j-1}}  \\
\end{array}} ;x\!\right) \end{array}  \right\}.
\end{align}
\end{thm}
\begin{proof}
First, observe that
\begin{align*}
&\xi^{(x)}_{z_0}\left(\!\! {\begin{array}{*{20}{c}}
   {\{z_1\}_{q-1},\bfz_r}  \\
   {\{1\}_{q-1},\overrightarrow{\bfk}_r^{-}}  \\
\end{array}} ;{\begin{array}{*{20}{c}}
   {\{z_{r+1}\}_p}  \\
  \{1\}_p  \\
\end{array}}\!\!\right)=\int_0^x L_{z_0}\left(\!\! {\begin{array}{*{20}{c}}
   {\{z_{r+1}\}_p}  \\
   {{\{1\}_p}}  \\
\end{array}} ;t\!\right) z_0(t) L_{z_0}\left(\!\! {\begin{array}{*{20}{c}}
   {\{z_1\}_{q-1},\bfz_r}  \\
   {\{1\}_{q-1},\overrightarrow{\bfk}_r^{-}}  \\
\end{array}} ;t\!\right) dt.
\end{align*}
Applying \eqref{diff-RF} and integrating by parts, we obtain
\begin{align*}
&\xi^{(x)}_{z_0}\left(\!\! {\begin{array}{*{20}{c}}
   {\{z_1\}_{q-1},\bfz_r}  \\
   {\{1\}_{q-1},\overrightarrow{\bfk}_r^{-}}  \\
\end{array}} ;{\begin{array}{*{20}{c}}
   {\{z_{r+1}\}_p}  \\
  \{1\}_p  \\
\end{array}}\!\!\right)\nonumber\\
&=\int_0^x L_{z_0}\left(\!\! {\begin{array}{*{20}{c}}
   {\{z_1\}_{q-1},\bfz_r}  \\
   {\{1\}_{q-1},\overrightarrow{\bfk}_r^{-}}  \\
\end{array}} ;t\!\right)dL_{z_0}\left(\!\! {\begin{array}{*{20}{c}}
   {\{z_{r+1}\}_{p}}  \\
   {{\{1\}_{p-1}},2}  \\
\end{array}} ;t\!\right)\\
&=L_{z_0}\left(\!\! {\begin{array}{*{20}{c}}
   {\{z_{r+1}\}_{p}}  \\
   {{\{1\}_{p-1}},2}  \\
\end{array}} ;x\!\right) L_{z_0}\left(\!\! {\begin{array}{*{20}{c}}
   {\{z_1\}_{q-1},\bfz_r}  \\
   {\{1\}_{q-1},\overrightarrow{\bfk}_r^{-}}  \\
\end{array}} ;x\!\right)\\
&\quad-\int_0^x L_{z_0}\left(\!\! {\begin{array}{*{20}{c}}
   {\{z_{r+1}\}_{p}}  \\
   {{\{1\}_{p-1}},2}  \\
\end{array}} ;t\!\right)z_0(t) L_{z_0}\left(\!\! {\begin{array}{*{20}{c}}
   {\{z_1\}_{q-1},\bfz_r}  \\
   {\{1\}_{q-1},k_1,\ldots,k_{r-1},k_r-2}  \\
\end{array}} ;t\!\right)dt\\
&=\cdots\\
&=\sum_{i=1}^{k_r-2}(-1)^{i-1} L_{z_0}\left(\!\! {\begin{array}{*{20}{c}}
   {\{z_{r+1}\}_{p}}  \\
   {{\{1\}_{p-1}},i+1}  \\
\end{array}} ;x\!\right) L_{z_0}\left(\!\! {\begin{array}{*{20}{c}}
   {\{z_1\}_{q-1},\bfz_r}  \\
   {\{1\}_{q-1},k_1,\ldots,k_{r-1},k_r-i}  \\
\end{array}} ;x\!\right) \\
&\quad+(-1)^{k_r-2} \int_0^x  L_{z_0}\left(\!\! {\begin{array}{*{20}{c}}
   {\{z_{r+1}\}_{p}}  \\
   {{\{1\}_{p-1}},k_r-1}  \\
\end{array}} ;t\!\right) z_0(t) L_{z_0}\left(\!\! {\begin{array}{*{20}{c}}
   {\{z_1\}_{q-1},\bfz_r}  \\
   {\{1\}_{q-1},k_1,\ldots,k_{r-1},1}  \\
\end{array}} ;t\!\right) dt\\
&=\sum_{i=1}^{k_r-2}(-1)^{i-1} L_{z_0}\left(\!\! {\begin{array}{*{20}{c}}
   {\{z_{r+1}\}_{p}}  \\
   {{\{1\}_{p-1}},i+1}  \\
\end{array}} ;x\!\right) L_{z_0}\left(\!\! {\begin{array}{*{20}{c}}
   {\{z_1\}_{q-1},\bfz_r}  \\
   {\{1\}_{q-1},k_1,\ldots,k_{r-1},k_r-i}  \\
\end{array}} ;x\!\right) \\
&\quad+(-1)^{k_r}\left\{\begin{array}{l}  L_{z_0}\left(\!\! {\begin{array}{*{20}{c}}
   {\{z_{r+1}\}_{p}}  \\
   {{\{1\}_{p-1}},k_r}  \\
\end{array}} ;x\!\right) L_{z_0}\left(\!\! {\begin{array}{*{20}{c}}
   {\{z_1\}_{q-1},\bfz_r}  \\
   {\{1\}_{q-1},k_1,\ldots,k_{r-1},1}  \\
\end{array}} ;x\!\right) \\ - L_{z_0}\left(\!\! {\begin{array}{*{20}{c}}
   {\{z_{r+1}\}_{p},z_r}  \\
   {{\{1\}_{p-1}},k_r,1}  \\
\end{array}} ;x\!\right) L_{z_0}\left(\!\! {\begin{array}{*{20}{c}}
   {\{z_1\}_{q-1},\overrightarrow{\bfz}_{r-1}}  \\
   {\{1\}_{q-1},k_1,\ldots,k_{r-1}}  \\
\end{array}} ;x\!\right) \end{array}\right\}\\
&\quad+(-1)^{k_r} \int_0^x L_{z_0}\left(\!\! {\begin{array}{*{20}{c}}
   {\{z_{r+1}\}_p,z_r}  \\
   {{\{1\}_{p-1}},k_r,1}  \\
\end{array}} ;t\!\right)z_0(t) L_{z_0}\left(\!\! {\begin{array}{*{20}{c}}
   {\{z_1\}_{q-1},\overrightarrow{\bfz}_{r-1}}  \\
   {\{1\}_{q-1},\overrightarrow{\bfk}_{r-2},k_{r-1}-1}  \\
\end{array}} ;t\!\right)dt.
\end{align*}
Repeating this process $r-1$ times yields the desired formula.
\end{proof}

If we take $z_{j}(t)=\frac{1}{1-t}$ or $z_{j}(t)=\frac{2}{1-t^2}$ for $j=1,\ldots,r+1$ and $z_0(t)=\frac1{t}$ with $x\to 1$, we recover the general duality relations for Arakawa-Kaneko zeta values and Kaneko-Tsumura $\psi$-values; see~\cite[Proposition 3.1]{X2021} and~\cite[Theorem 3.6]{PX2019}.

We note that for any $x$ for which $\int_0^{x} z(u)du$ converges, the shuffle relations give
\begin{align}\label{equ-L-TRAN}
L_{z_0}\left(\!\! {\begin{array}{*{20}{c}}
   {\{z\}_p}  \\
   {{\{1\}_p}}  \\
\end{array}} ;t\!\right)
=&\,  \frac1{p!} \Big(\int_0^x  z(u)du -\int_t^x  z(u)du \Big)^p \notag\\
=&\, \sum_{j=0}^p (-1)^j \frac{\Big(\int_0^{x} z(u)du\Big)^{p-j}}{(p-j)!}\int_t^x \Big(z(u)du\Big)^j.
\end{align}
Applying \eqref{equ-L-TRAN} yields
\begin{align}\label{equ-L-Fun-relation}
\xi^{(x)}_{z_0}\left(\!\! {\begin{array}{*{20}{c}}
   {\{z_1\}_{q-1},\bfz_r}  \\
   {\{1\}_{q-1},\overrightarrow{\bfk}_r^{-}}  \\
\end{array}} ;{\begin{array}{*{20}{c}}
   {\{z\}_p}  \\
  \{1\}_p  \\
\end{array}}\!\!\right)
=\sum_{j=0}^p (-1)^j \frac{\Big(\int_0^{x} z(u)du\Big)^{p-j}}{(p-j)!}  L_{z_0}\left(\!\! {\begin{array}{*{20}{c}}
   {\{z_1\}_{q-1},\bfz_r,\{z\}_j}  \\
   {\{1\}_{q-1},\bfk_r,\{1\}_j}  \\
\end{array}} ;x\!\right).
\end{align}
Hence, substituting \eqref{equ-L-Fun-relation} with $z=z_{r+1}$ into \eqref{GKTSAKF-DUALT-NOEXP} gives the following corollary.

\begin{cor} For positive integers $p,q,r$ and a composition ${\bfk}=(k_1,\ldots,k_r)$ with $k_1,k_2,\ldots,k_r\in \N\setminus\{1\}$,
\begin{align}\label{GKTSAKF-DUALT-EXP}
&\sum_{j=0}^p (-1)^j \frac{\Big(\int_0^{x} z_{r+1}(u)du\Big)^{p-j}}{(p-j)!}  L_{z_0}\left(\!\! {\begin{array}{*{20}{c}}
   {\{z_1\}_{q-1},\bfz_r,\{z_{r+1}\}_j}  \\
   {\{1\}_{q-1},\bfk_r,\{1\}_j}  \\
\end{array}} ;x\!\right)
\nonumber\\&\quad-(-1)^{|{\bfk}|}\sum_{j=0}^q (-1)^j \frac{\Big(\int_0^{x} z_{1}(u)du\Big)^{q-j}}{(q-j)!} L_{z_0}\left(\!\! {\begin{array}{*{20}{c}}
   {\{z_{r+1}\}_{p},z_r,\ldots,z_2,\{z_{1}\}_j}  \\
   {\{1\}_{p-1},\overleftarrow{\bfk}_r,\{1\}_j}  \\
\end{array}} ;x\!\right)\nonumber\\
&=\sum\limits_{j=0}^{r-1} (-1)^{\mid\stackrel{\leftarrow}{{\bfk}}_j\mid}\sum\limits_{i=1}^{k_{r-j}-2}(-1)^{i-1} L_{z_0}\left(\!\! {\begin{array}{*{20}{c}}
   {\{z_{r+1}\}_p,\overleftarrow{\bfz_j}}  \\
   {\{1\}_{p-1},{\overleftarrow{\bfk}_j},i+1}  \\
\end{array}} ;x\!\right)L_{z_0}\left(\!\! {\begin{array}{*{20}{c}}
   {\{z_1\}_{q-1},{\overrightarrow{\bfz}}_{r-j}}  \\
   {\{1\}_{q-1},\overrightarrow{\bfk}_{r-j-1},k_{r-j}-i}  \\
\end{array}} ;x\!\right)   \nonumber\\
&\quad+\sum\limits_{j=0}^{r-2}(-1)^{\mid\stackrel{\leftarrow}{{\bfk}}_{j+1}\mid} \left\{\begin{array}{l} L_{z_0}\left(\!\! {\begin{array}{*{20}{c}}
   {\{z_{r+1}\}_p,{\overleftarrow{\bfz}}_j}  \\
   {\{1\}_{p-1},{\overleftarrow{\bfk}_{j+1}}}  \\
\end{array}} ;x\!\right) L_{z_0}\left(\!\! {\begin{array}{*{20}{c}}
   {\{z_1\}_{q-1},{\overrightarrow{\bfz}}_{r-j}}  \\
   {\{1\}_{q-1},{\overrightarrow{\bfk}}_{r-j-1},1}  \\
\end{array}} ;x\!\right)   \\
\quad\quad-L_{z_0}\left(\!\! {\begin{array}{*{20}{c}}
   {\{z_{r+1}\}_p,{\overleftarrow{\bfz}}_{j+1}}  \\
   {\{1\}_{p-1},{\overleftarrow{\bfk}_{j+1}},1}  \\
\end{array}} ;x\!\right) L_{z_0}\left(\!\! {\begin{array}{*{20}{c}}
   {\{z_1\}_{q-1},{\overrightarrow{\bfz}}_{r-j-1}}  \\
   {\{1\}_{q-1},{\overrightarrow{\bfk}}_{r-j-1}}  \\
\end{array}} ;x\!\right) \end{array}  \right\}.
\end{align}
\end{cor}

Specializing $z_{j}(t)=\frac{1}{1+t}$ for all $j=1,\ldots,r+1$ and $z_0(t)=\frac1{t}$ with $x\to 1$ in \eqref{GKTSAKF-DUALT-EXP} gives the following corollary.

\begin{cor}
For positive integers $p,q,r$ and a composition $\bfk=(k_1,\ldots,k_r)$ with $k_1,k_2,\ldots,k_r\in \N\setminus\{1\}$,
\begin{align}\label{thm-duality-formula-generalAMZVs}
&\sum_{j=0}^p (-1)^j \frac{\log^{p-j}(2)}{(p-j)!}  \barz \Big(\{1\}_{q-1},\bfk_r,\{1\}_j\Big)- (-1)^{|{\bfk}|}\sum_{j=0}^q (-1)^j \frac{\log^{q-j}(2)}{(q-j)!} \barz \Big( \{1\}_{p-1},\overleftarrow{\bfk}_r,\{1\}_j\Big)\nonumber\\
&=\sum\limits_{j=0}^{r-1} (-1)^{\mid\stackrel{\leftarrow}{{\bfk}}_j\mid}\sum\limits_{i=1}^{k_{r-j}-2}(-1)^{i-1} \barz \Big(\{1\}_{p-1},{\overleftarrow{\bfk}_j},i+1\Big) \barz \Big(\{1\}_{q-1},\overrightarrow{\bfk}_{r-j-1},k_{r-j}-i\Big)  \nonumber\\
&\quad+\sum\limits_{j=0}^{r-2}(-1)^{\mid\stackrel{\leftarrow}{{\bfk}}_{j+1}\mid} \left\{\begin{array}{l} \barz \Big(\{1\}_{p-1},{\overleftarrow{\bfk}_{j+1}}\Big) \barz \Big(\{1\}_{q-1},{\overrightarrow{\bfk}}_{r-j-1},1\Big)  \\
\quad\quad-  \barz \Big(\{1\}_{p-1},{\overleftarrow{\bfk}_{j+1}},1\Big) \barz \Big(\{1\}_{q-1},{\overrightarrow{\bfk}}_{r-j-1}\Big)\end{array}  \right\},
\end{align}
where for any positive integers $k_1,\dotsc,k_r$,
\begin{align*}
\barz(k_1,\dotsc,k_{r-1},{k_r}):=\sum\limits_{0<n_1<\cdots<n_r} \frac{(-1)^{n_r}}{n_1^{k_1}\dotsm n_r^{k_r}}
=\int_0^1 \frac{dt}{1+t}\left(\!\frac{dt}{t}\!\right)^{k_1-1}
\cdots\frac{dt}{1+t}\left(\!\frac{dt}{t}\!\right)^{k_r-1}.
\end{align*}
\end{cor}

Similarly, putting $z_{j}(t)=\frac{2}{1+t^2}$ for all $j=1,\ldots,r+1$ and $z_0(t)=\frac1{t}$ with $x\to 1$ in \eqref{GKTSAKF-DUALT-EXP} yields the following corollary.

\begin{cor}
For positive integers $p,q,r$ and a composition ${\bfk}=(k_1,\ldots,k_r)$ with $k_1,k_2,\ldots,k_r\in \N\setminus\{1\}$,
\begin{align*}
&(-1)^{q+r-1} \sum_{j=0}^{p} \frac{\left(\!\frac{\pi}{2}\!\right)^{p-j}}{(p-j)!}\overline{T}\left(\!\{1\}_{q-1},\bfk_r,\{1\}_j\!\right)-(-1)^{|{\bfk}|}(-1)^{p+r-1}\sum_{j=0}^{q} \frac{\left(\!\frac{\pi}{2}\!\right)^{q-j}}{(q-j)!}\overline{T}\left(\!\{1\}_{p-1},\overleftarrow{\bfk_r},\{1\}_j\!\right)\\
&=(-1)^{p+q+r}\sum\limits_{j=0}^{r-1} (-1)^{\mid\stackrel{\leftarrow}{{\bfk}}_j\mid}\sum\limits_{i=1}^{k_{r-j}-2} (-1)^{i} \overline{T}\left(\!\{1\}_{p-1},\overleftarrow{\bfk}_j,i+1\!\right)\overline{T}\left(\!\{1\}_{q-1},\overrightarrow{\bfk}_{r-j-1},k_{r-j}-i\!\right)\\
&\quad-(-1)^{p+q+r}\sum\limits_{j=0}^{r-2} (-1)^{\mid\stackrel{\leftarrow}{{\bfk}}_{j+1}\mid} \left\{\begin{array}{l} \overline{T}\left(\!\{1\}_{p-1},\overleftarrow{\bfk}_{j+1}\!\right)\overline{T}\left(\!\{1\}_{q-1},\overrightarrow{\bfk}_{r-j-1},1\!\right)\\ -\overline{T}\left(\!\{1\}_{p-1},\overleftarrow{\bfk}_{j+1},1\!\right)\overline{T}\left(\!\{1\}_{q-1},\overrightarrow{\bfk}_{r-j-1}\!\right)\end{array}\right\},
\end{align*}
where for any positive integers $k_1,\dotsc,k_r$,
\begin{align*}
\overline{T}(k_1,\dotsc,k_{r-1},{k_r}):&=2^r\sum\limits_{0<n_1<\cdots<n_r} \frac{(-1)^{n_r}}{(2n_1-1)^{k_1}\dotsm (2n_{r-1}-r+1)^{k_{r-1}}(2n_r-r)^{k_r}}\nonumber\\
&=(-1)^r\int_0^1 \frac{2dt}{1+t^2}\left(\!\frac{dt}{t}\!\right)^{k_1-1}
\cdots\frac{2dt}{1+t^2}\left(\!\frac{dt}{t}\!\right)^{k_r-1}.
\end{align*}
We call these \emph{multiple $\overline{T}$-values} \emph{(M$\overline{\rm T}$Vs)}.
\end{cor}

\begin{thm} For $a,b\in \N_0:=\N\cup \{0\}$ and $z=z(t)$,
\begin{align}\label{equ-aLFR}
L_{z_0}\left(\!\!  {\begin{array}{*{20}{c}}
   {\{z\}_{a+b+1}}  \\
   {{\{1\}_a},2,\{1\}_{b}}  \\
\end{array}} ;x\! \!\right)=\sum_{k=0}^{b} (-1)^k \frac{\left(\!\! L_{z_0}\left(\!\! {\begin{array}{*{20}{c}}
   {z}  \\
   {1}  \\
\end{array}} ;x\!\right)\! \!\right)^{b+k}}{(b-k)!}\binom{a+1+k}{k} L_{z_0}\left(\!\!  {\begin{array}{*{20}{c}}
   {\{z\}_{a+1+k}}  \\
   {{\{1\}_{a+k}},2}  \\
\end{array}} ;x\! \!\right).
\end{align}
\end{thm}

\begin{proof}
We note that
\begin{align*}
&L_{z_0}\left(\!\! {\begin{array}{*{20}{c}}
   {\{z\}_{a+b+1}}  \\
   {{\{1\}_a},2,\{1\}_{b}}  \\
\end{array}} ;x \! \!\right)=\int_0^x \Big(z(t)dt\Big)^{a+1}z_0(t)dt\Big(z(t)dt\Big)^{b}\\
&=\frac{1}{(a+1)!b!}\int_0^x \left(\!\int_0^tz(u)du\!\right)^{a+1}z_0(t)\left(\!\int_0^x z(u)du-\int_0^t z(u)du\!\right)^{b}dt\\
&=\frac{1}{(a+1)!b!} \sum_{k=0}^{b} (-1)^k \binom{b+1}{k} \left(\!\int_0^x z(u)du\!\right)^{b-k}\int_0^x \left(\!\int_0^t z(u)du\!\right)^{a+1+k}z_0(t)dt\\
&=\sum_{k=0}^{b}\frac{ (-1)^k (a+1+k)!}{(a+1)!b!}  \binom{b+1}{k} \left(\!L_{z_0}\left(\!\! {\begin{array}{*{20}{c}}
   {z}  \\
   {1}  \\
\end{array}} ;x\!\right)\!\right)^{b+k}\int_0^x \big(z(t)dt\big)^{a+1+k}z_0(t)dt.
\end{align*}
Finally, by the definition of the formal multiple polylogarithm function, we have
\begin{align*}
\int_0^x \big(z(t)dt\big)^{a+1+k}z_0(t)dt=L_{z_0}\left(\!\! {\begin{array}{*{20}{c}}
   {\{z\}_{a+1+k}}  \\
   {{\{1\}_{a+k}},2}  \\
\end{array}} ;x\!\right).
\end{align*}
This completes the proof.
\end{proof}

Taking $z(t)=\frac{1}{1+t}$ and $z_0(t)=\frac1{t}$ with $x\to 1$ in \eqref{equ-aLFR} gives (see also~\cite[Thm. 5.1]{Xu2019})

\begin{cor}
For $a,b\in \N_0$,
\begin{align*}
\barz({\{1\}_a},2,\{1\}_b)=\sum_{k=0}^{b} (-1)^k \frac{\log^{b-k}(2)}{(b-k)!}\binom{a+1+k}{k}\barz ({\{1\}_{a+k}},2).
\end{align*}
\end{cor}
Similarly, setting $z(t)=\frac{2}{1+t^2}$ and $z_0(t)=\frac1{t}$ with $x\to 1$ in \eqref{equ-aLFR} gives the result in~\cite[Thm 4.15]{XYZ2022}.

\section{Sum Formulas for Formal Arakawa-Kaneko Zeta Values}

In this section, we establish several sum formulas involving the formal multiple polylogarithm function and the general parametric Arakawa-Kaneko zeta values, again using iterated integrals.

\begin{thm}\label{thm-Relation-FMPLs}
For any composition $\bfk=(k_1,\ldots,k_r)$ and $k,r\in \N$, we have
\begin{align}
&\sum\limits_{k_1+\cdots+k_r=k+r-1\atop k_1,\ldots,k_r\geq 1} L_{z_0}\left(\!\! {\begin{array}{*{20}{c}}
   {\{z\}_{r}}  \\
   {\bfk}  \\
\end{array}} ;x\!\right)=\sum_{j=1}^r (-1)^{j-1} L_{z_0}\left(\!\! {\begin{array}{*{20}{c}}
   {\{z\}_{r-j}}  \\
   {{\{1\}_{r-j}}}  \\
\end{array}} ;x\!\right)L_{z_0}\left(\!\! {\begin{array}{*{20}{c}}
   {\{z\}_{j}}  \\
   {{\{1\}_{j-1}},k}  \\
\end{array}} ;x\!\right),\label{sum-equ-one-special}\\
&L_{z_0}\left(\!\! {\begin{array}{*{20}{c}}
   {\{z\}_{r}}  \\
   {{\{1\}_{r-1}},k}  \\
\end{array}} ;x\!\right)=\sum_{j=1}^r (-1)^{j-1} L_{z_0}\left(\!\! {\begin{array}{*{20}{c}}
   {\{z\}_{r-j}}  \\
   {{\{1\}_{r-j}}}  \\
\end{array}} ;x\!\right) \sum\limits_{k_1+\cdots+k_j=k+j-1\atop k_1,\ldots,k_j\geq 1} L_{z_0}\left(\!\! {\begin{array}{*{20}{c}}
   {\{z\}_{j}}  \\
   {\overrightarrow{\bfk}_j}  \\
\end{array}} ;x\!\right).\label{sum-equ-two-special}
\end{align}
\end{thm}
\begin{proof} Using the elementary identity
\begin{align*}
\int_{x}^y \big(z(t)dt\big)^{k}=\frac{1}{k!} \Big(\int_x^y z(t)dt\Big)^k
\end{align*}
and \eqref{equ-defn-fmplf}, we obtain (with $t_{r+1}:=x$)
\begin{align}\label{equ-iterexperss-FMPLs}
L_{z_0}\left(\!\! {\begin{array}{*{20}{c}}
   {\{z\}_{r}}  \\
   {\bfk}  \\
\end{array}} ;x\!\right)=\frac1{\prod\limits_{j=1}^r (k_j-1)!}\int_{0<t_1<\cdots<t_{r}<x} \left(\!\prod\limits_{j=1}^r z(t_j) \Big(\int_{t_{j}}^{t_{j+1}}z_0(t) dt\Big)^{k_j-1}\!\right)\prod\limits_{j=1}^rdt_j.
\end{align}
Summing over all compositions $\bfk$ with $|\bfk|=k+r-1$ gives
\begin{align}\label{sum-equ-FMPLS-iter}
&\sum\limits_{|\bfk|=k+r-1\atop k_1,\ldots,k_r\geq 1} L_{z_0}\left(\!\! {\begin{array}{*{20}{c}}
   {\{z\}_{r}}  \\
   {\bfk}  \\
\end{array}} ;x\!\right)=\frac1{(k-1)!} \int_{0<t_1<\cdots<t_{r}<x}  \left(\!\prod\limits_{j=1}^r z(t_j)\!\right)\left(\!  \sum_{j=1}^r \int_{t_{j}}^{t_{j+1}} z_0(t)dt\!\right)^{k-1} \prod\limits_{j=1}^rdt_j\nonumber\\
&=\frac1{(k-1)!} \int_{0<t_1<\cdots<t_{r}<x}  \left(\!\prod\limits_{j=1}^r z(t_j)\!\right) \left(\!\int_{t_{1}}^{x} z_0(t)dt\!\right)^{k-1}\prod\limits_{j=1}^rdt_j\nonumber\\
&=\frac1{(k-1)!} \int_{0}^x \left(\!\int_{t_{1}}^{x} z_0(t)dt\!\right)^{k-1} z(t_1)\left(\!\int_{t_1<t_2<\cdots<t_{r}<x} z(t_2)\cdots z(t_r)dt_2\cdots dt_r\!\right) dt_1\nonumber\\
&=\frac1{(k-1)!(r-1)!}   \int_{0}^x \left(\!\int_{t}^{x} z_0(u)du\!\right)^{k-1} z(t) \left(\!\int_{t}^{x} z(u)du\!\right)^{r-1}dt.
\end{align}
On the other hand, taking $\bfk=(\{1\}_{r-1},k)$ in \eqref{equ-iterexperss-FMPLs} yields
\begin{align}\label{nosum-equ-FMPLS-iter}
&L_{z_0}\left(\!\! {\begin{array}{*{20}{c}}
   {\{z\}_{r}}  \\
   {{\{1\}_{r-1}},k}  \\
\end{array}} ;x\!\right)=\frac1{(k-1)!} \int_{0<t_1<\cdots<t_{r}<x}  \left(\!\prod\limits_{j=1}^r z(t_j)\!\right)  \left(\!\int_{t_r}^x z_0(t)dt\!\right)^{k-1}\prod\limits_{j=1}^rdt_j\nonumber\\
&=\frac1{(k-1)!} \int_{0}^x \left(\!\int_{t_{r}}^{x} z_0(t)dt\!\right)^{k-1} z(t_r)\left(\!\int_{0<t_1<\cdots<t_{r-1}<t_t} z(t_1)\cdots z(t_{r-1})dt_1\cdots dt_{r-1}\!\right) dt_r\nonumber\\
&=\frac1{(k-1)!(r-1)!}  \int_{0}^x \left(\!\int_{t}^{x} z_0(u)du\!\right)^{k-1} z(t) \left(\!\int_{0}^{t} z(u)du\!\right)^{r-1}dt.
\end{align}
Combining \eqref{sum-equ-FMPLS-iter} and \eqref{nosum-equ-FMPLS-iter}, we arrive at
\begin{align*}
&\sum\limits_{|\bfk|=k+r-1\atop k_1,\ldots,k_r\geq 1} L_{z_0}\left(\!\! {\begin{array}{*{20}{c}}
   {\{z\}_{r}}  \\
   {\bfk}  \\
\end{array}} ;x\!\right)=\frac{\int_{0}^x \left(\!\int_{t}^{x} z_0(u)du\!\right)^{k-1} z(t) \left(\!\int_{0}^{x} z(u)du-\int_{0}^{t} z(u)du\!\right)^{r-1}dt}{(k-1)!(r-1)!}   \nonumber\\
&=\frac1{(k-1)!(r-1)!} \sum_{j=1}^r (-1)^{j-1} \binom{r-1}{j-1} \left(\!\int_0^x z(t) dt \!\right)^{r-j} \nonumber\\&\quad\quad\quad\quad\quad\quad\quad\quad\quad\quad\times\int_{0}^x \left(\!\int_{t}^{x} z_0(u)du\!\right)^{k-1} z(t) \left(\!\int_{0}^{t} z(u)du\!\right)^{j-1}dt\nonumber\\
&=\sum_{j=1}^r \frac{(-1)^{j-1}}{(r-j)!}\left(\!\int_0^x z(t) dt \!\right)^{r-j} L_{z_0}\left(\!\! {\begin{array}{*{20}{c}}
   {\{z\}_{j}}  \\
   {{\{1\}_{j-1}},k}  \\
\end{array}} ;x\!\right)
\end{align*}
and
\begin{align*}
&L_{z_0}\left(\!\! {\begin{array}{*{20}{c}}
   {\{z\}_{r}}  \\
   {{\{1\}_{r-1}},k}  \\
\end{array}} ;x\!\right)=\frac{  \int_{0}^x \left(\!\int_{t}^{x} z_0(u)du\!\right)^{k-1} z(t) \left(\!\int_{0}^{x} z(u)du-\int_{t}^{x} z(u)du\!\right)^{r-1}dt}{(k-1)!(r-1)!} \nonumber\\
&=\frac1{(k-1)!(r-1)!} \sum_{j=1}^r (-1)^{j-1} \binom{r-1}{j-1} \left(\!\int_0^x z(t) dt \!\right)^{r-j} \nonumber\\&\quad\quad\quad\quad\quad\quad\quad\quad\quad\quad\times\int_{0}^x \left(\!\int_{t}^{x} z_0(u)du\!\right)^{k-1} z(t) \left(\!\int_{t}^{x} z(u)du\!\right)^{j-1}dt\\
&=\sum_{j=1}^r \frac{(-1)^{j-1}}{(r-j)!}\left(\!\int_0^x z(t) dt \!\right)^{r-j} \sum\limits_{k_1+\cdots+k_j=k+j-1\atop k_1,\ldots,k_j\geq 1} L_{z_0}\left(\!\! {\begin{array}{*{20}{c}}
   {\{z\}_{j}}  \\
   {\overrightarrow{\bfk}_j}  \\
\end{array}} ;x\!\right).
\end{align*}
Finally, using
\begin{align*}
\left(\!\int_0^x z(t) dt \!\right)^{r}=r!L_{z_0} \left(\!\! {\begin{array}{*{20}{c}}
   {\{z\}_{r}}  \\
   {{\{1\}_{r}}}  \\
\end{array}} ;x\!\right),
\end{align*}
we obtain the desired identities.
\end{proof}

If we set $z(t)=\frac{1}{1-t}$ and $z_0(t)=\frac1{t}$ in Theorem \ref{thm-Relation-FMPLs}, and $z(t)=\frac{2}{1-t^2}$ and $z_0(t)=\frac1{t}$ with $x\to 1$, we get the following corollaries.

\begin{cor}\emph{(cf.~\cite[Eqs. (2.10)-(2.11)]{LuoSi2023})}
For positive integers $k$ and $r$ with $x\in (0,1)$,
\begin{align}\label{cc6}
\Li_{\{1\}_{r-1},k}(x)=(-1)^{r-1}\sum_{j=1}^r \frac{\log^{r-j}(1-x)}{(r-j)!} \sum_{k_1+\cdots+k_j=k+j-1,\atop k_1,\ldots,k_j\geq 1} \Li_{k_1,k_2,\ldots,k_j}(x)
\end{align}
and
\begin{align}\label{cc7}
\sum_{k_1+\cdots+k_r=k+r-1,\atop k_1,\ldots,k_r\geq 1} \Li_{k_1,k_2,\ldots,k_r}(x)=(-1)^{r-1}\sum_{j=1}^r  \frac{\log^{r-j}(1-x)}{(r-j)!} \Li_{\{1\}_{j-1},k}(x).
\end{align}
\end{cor}

When $r=2$, either \eqref{cc6} or \eqref{cc7} gives (using the multi-variable multiple polylogarithm notation)
\begin{equation*}
 \log(1-x)\Li_{k}(x)+\Li_{1,k}(x)+\sum_{m+n=k+1,\atop m,n\geq 1} \Li_{m,n}(1,x)=0.
\end{equation*}
Since $\log(1-x)=-\Li_1(x)$, the stuffle relation $\Li_1(x)\Li_{k}(x)=\Li_{1,k}(x,x)+\Li_{k,1}(x,x)+\Li_{k+1}(x)$ implies
\begin{equation*}
\sum_{m+n=k+1,\atop m\geq 1, n\geq 2} \Li_{m,n}(x)=\Li_{k+1}(x)+\big(\Li_{k,1}(x,x)-\Li_{k,1}(1,x)\big)
+\big(\Li_{1,k}(x,x)-\Li_{1,k}(1,x)\big).
\end{equation*}
Letting $x\to 1$ recovers the classical sum formula for double zeta values:
\begin{equation*}
\sum_{m+n=k+1,\atop m\geq 1, n\geq 2} \zeta(m,n)=\zeta(k+1).
\end{equation*}
Based on our computations, we pose the following question:

\begin{qu}
Can \eqref{cc6} or \eqref{cc7} be used to prove the well-known sum formula for multiple zeta values? Namely,
\[\sum_{k_1+\cdots+k_r=k+1,\atop k_1,\ldots,k_{r-1}\geq 1,k_r\geq 2} \zeta(k_1,\ldots,k_r)=\zeta(k+1)\quad (k\in \N).\]
\end{qu}

\begin{cor} For positive integers $k$ and $r$ with $x\in (0,1)$,
\begin{align}\label{pro-formula-KTA1}
\sum\limits_{k_1+\cdots+k_r=k+r-1\atop k_1,\ldots,k_r\geq 1} \A\left(\!k_1,\ldots,k_r;x\!\right)=(-1)^{r-1}\sum\limits_{j=1}^{r} \frac{\log^{r-j}\left(\!\frac{1-x}{1+x}\!\right) }{(r-j)!}\A\left(\!\{1\}_{j-1},k;z\!\right)
\end{align}
and
\begin{align}\label{pro-formula-KTA2}
\A\left(\!\{1\}_{r-1},k;x\!\right)=(-1)^{r-1} \sum_{j=1}^{r} \frac{\log^{r-j}\left(\!\frac{1-x}{1+x}\!\right)}{(r-j)!} \sum_{k_1+\cdots+k_{j}=k+j-1,\atop k_1,\ldots,k_{j}\geq 1} \A\left(\!k_1,\ldots,k_{j};x\!\right).
\end{align}
\end{cor}

\begin{thm}\label{thm-GKTPsi-sum-dual} For any composition $\bfk=(k_1,\ldots,k_r)$ and integer $p\in \N_0$, we have
\begin{align*}
\xi^{(x)}_{z_0}\left(\!\! {\begin{array}{*{20}{c}}
   {\{z\}_r}  \\
   {\{1\}_{r-1},k}  \\
\end{array}} ;{\begin{array}{*{20}{c}}
   {\{z\}_p}  \\
  \{1\}_p \\
\end{array}}\!\!\right)=\sum_{j=1}^r \sum_{|\overrightarrow{\bfk}_j|=k+j-1,\atop k_1,\ldots,k_j\geq 1} (-1)^{j-1} \binom{p+r-j}{p} \xi^{(x)}_{z_0}\left(\!\! {\begin{array}{*{20}{c}}
   {\{z\}_j}  \\
   {\overrightarrow{\bfk}_j}  \\
\end{array}} ;{\begin{array}{*{20}{c}}
   {\{z\}_{p+r-j}}  \\
  \{1\}_{p+r-j} \\
\end{array}}\!\!\right)
\end{align*}
and
\begin{align*}
\sum_{|\bfk|=k+r-1,\atop k_,\ldots,k_r\geq 1} \xi^{(x)}_{z_0}\left(\!\! {\begin{array}{*{20}{c}}
   {\{z\}_r}  \\
   {\overrightarrow{\bfk}}  \\
\end{array}} ;{\begin{array}{*{20}{c}}
   {\{z\}_{p}}  \\
  \{1\}_{p} \\
\end{array}}\!\!\right)=\sum_{j=1}^r (-1)^{j-1} \binom{p+r-j}{p}  \xi^{(x)}_{z_0}\left(\!\! {\begin{array}{*{20}{c}}
   {\{z\}_j}  \\
   {\{1\}_{j-1},k}  \\
\end{array}} ;{\begin{array}{*{20}{c}}
   {\{z\}_{p+r-j}}  \\
  \{1\}_{p+r-j} \\
\end{array}}\!\!\right).
\end{align*}
\end{thm}

\begin{proof}
By the definition of $\xi^{(x)}_{z_0}(\cdot)$ and using \eqref{sum-equ-one-special}-\eqref{sum-equ-two-special}, we get
\begin{align*}
&\xi^{(x)}_{z_0}\left(\!\! {\begin{array}{*{20}{c}}
   {\{z\}_r}  \\
   {\{1\}_{r-1},k}  \\
\end{array}} ;{\begin{array}{*{20}{c}}
   {\{z\}_p}  \\
  \{1\}_p \\
\end{array}}\!\!\right)=\int_0^x L_{z_0}\left(\!\! {\begin{array}{*{20}{c}}
   {\{z\}_p}  \\
   {{\{1\}_p}}  \\
\end{array}} ;t\!\right)z_0(t) L_{z_0}\left(\!\! {\begin{array}{*{20}{c}}
   {\{z\}_r}  \\
   {{\{1\}_{r-1},k}}  \\
\end{array}} ;t\!\right)dt\\
&=\sum_{j=1}^r \sum_{|\overrightarrow{\bfk}_j|=k+j-1,\atop k_1,\ldots,k_j\geq 1} (-1)^{j-1} \binom{p+r-j}{p}\int_0^x L_{z_0}\left(\!\! {\begin{array}{*{20}{c}}
   {\{z\}_{p+r-j}}  \\
   {{\{1\}_{p+r-j}}}  \\
\end{array}} ;t\!\right)z_0(t) L_{z_0}\left(\!\! {\begin{array}{*{20}{c}}
   {\{z\}_j}  \\
   {\overrightarrow{\bfk}_j}  \\
\end{array}} ;t\!\right)dt
\end{align*}
and
\begin{align*}
&\sum_{|\bfk|=k+r-1,\atop k_,\ldots,k_r\geq 1} \xi^{(x)}_{z_0}\left(\!\! {\begin{array}{*{20}{c}}
   {\{z\}_r}  \\
   {\overrightarrow{\bfk}}  \\
\end{array}} ;{\begin{array}{*{20}{c}}
   {\{z\}_{p}}  \\
  \{1\}_{p} \\
\end{array}}\!\!\right)=\sum_{|\bfk|=k+r-1,\atop k_,\ldots,k_r\geq 1} \int_0^x L_{z_0}\left(\!\! {\begin{array}{*{20}{c}}
   {\{z\}_{p}}  \\
   {{\{1\}_{p}}}  \\
\end{array}} ;t\!\right)z_0(t) L_{z_0}\left(\!\! {\begin{array}{*{20}{c}}
   {\{z\}_r}  \\
   {\bfk}  \\
\end{array}} ;t\!\right)dt \\
&=\sum_{j=1}^r  (-1)^{j-1} \binom{p+r-j}{p}\int_0^x L_{z_0}\left(\!\! {\begin{array}{*{20}{c}}
   {\{z\}_{p+r-j}}  \\
   {{\{1\}_{p+r-j}}}  \\
\end{array}} ;t\!\right)z_0(t) L_{z_0}\left(\!\! {\begin{array}{*{20}{c}}
   {\{z\}_j}  \\
   {{\{1\}_{j-1},k}}  \\
\end{array}} ;t\!\right)dt,
\end{align*}
where we used the shuffle relation
\begin{align*}
L_{z_0}\left(\!\! {\begin{array}{*{20}{c}}
   {\{z\}_{p}}  \\
   {{\{1\}_{p}}}  \\
\end{array}} ;t\!\right) L_{z_0}\left(\!\! {\begin{array}{*{20}{c}}
   {\{z\}_{r}}  \\
   {{\{1\}_{r}}}  \\
\end{array}} ;t\!\right)=\binom{p+r}{p} L_{z_0}\left(\!\! {\begin{array}{*{20}{c}}
   {\{z\}_{p+r}}  \\
   {{\{1\}_{p+r}}}  \\
\end{array}} ;t\!\right).
\end{align*}
This proves the theorem.
\end{proof}

Taking $z(t)=\frac{1}{1-t}$ and $z_0(t)=\frac1{t}$ (or $z(t)=\frac{2}{1-t^2}$ with $x\to 1$) in Theorem \ref{thm-GKTPsi-sum-dual} yields the following corollaries.

\begin{cor}\emph{(cf.~\cite[Thm. 1.1]{LuoSi2023})} \label{thm-AKTPsi-sum-dual}
For any composition $\bfk=(k_1,\ldots,k_r)$ and integer $p\in \N_0$, we have
\begin{align*}
\xi (\{1\}_{r-1},k;p+1)=\sum_{j=1}^r \sum_{k_1+\cdots+k_j=k+j-1,\atop k_1,\ldots,k_j\geq 1} (-1)^{j-1} \binom{p+r-j}{p} \xi (k_1,k_2,\ldots,k_j;p+1+r-j)
\end{align*}
and
\begin{align*}
\sum_{k_1+\cdots+k_r=k+r-1,\atop k_,\ldots,k_r\geq 1} \xi (k_1,k_2,\ldots,k_r;p+1)=\sum_{j=1}^r (-1)^{j-1}\binom{p+r-j}{p} \xi (\{1\}_{j-1},k;p+1+r-j).
\end{align*}
\end{cor}

\begin{cor}\label{thm-KTPsi-sum-dual}
For any composition $\bfk=(k_1,\ldots,k_r)$ and integer $p\in \N_0$, we have
\begin{align*}
\psi (\{1\}_{r-1},k;p+1)=\sum_{j=1}^r \sum_{k_1+\cdots+k_j=k+j-1,\atop k_1,\ldots,k_j\geq 1} (-1)^{j-1} \binom{p+r-j}{p} \psi (k_1,k_2,\ldots,k_j;p+1+r-j),
\end{align*}
and
\begin{align*}
\sum_{k_1+\cdots+k_r=k+r-1,\atop k_,\ldots,k_r\geq 1} \psi (k_1,k_2,\ldots,k_r;p+1)=\sum_{j=1}^r (-1)^{j-1}\binom{p+r-j}{p} \psi (\{1\}_{j-1},k;p+1+r-j).
\end{align*}
\end{cor}

To conclude this section, we present the following more general theorem on symmetric sum formulas.

\begin{thm}\label{thm-add-one}
For any composition $\bfk=(k_1,\ldots,k_r)$ and $k,r,m\in \N$, we have
\begin{align}
&\sum_{k_1+\cdots+k_r=k+r-1,\atop k_1,\ldots,k_r\geq1} \binom{k_r+m-2}{m-1} L_{z_0}\left(\!\! {\begin{array}{*{20}{c}}
   {{z_1}, \ldots,z_{r-1},{z_r}}  \\
   {k_1,\ldots,k_{r-1},k_r+m-1}  \\
\end{array}} ;x\!\right)\nonumber\\
&+\sum_{k_1+\cdots+k_r=m+r-1,\atop k_1,\ldots,k_r\geq1} \binom{k_r+k-2}{k-1} L_{z_0}\left(\!\! {\begin{array}{*{20}{c}}
   {{z_1}, \ldots,z_{r-1},{z_r}}  \\
   {k_1,\ldots,k_{r-1},k_r+k-1}  \\
\end{array}} ;x\!\right)\nonumber\\
&=\sum_{j=1}^{r-1} (-1)^{j-1}L_{z_0}\left(\!\! {\begin{array}{*{20}{c}}
   {{z_j}, \ldots,z_{2},{z_1}}  \\
   {1,\ldots,1,k}  \\
\end{array}} ;x\!\right)L_{z_0}\left(\!\! {\begin{array}{*{20}{c}}
   {{z_{j+1}}, \ldots,z_{r-1},{z_r}}  \\
   {1,\ldots,1,m}  \\
\end{array}} ;x\!\right).
\end{align}
\end{thm}
\begin{proof}
We begin by noting that, directly from Definition~\ref{equ-defn-fmplf}, a straightforward calculation gives
\begin{align}
L_{z_0}\left(\!\! {\begin{array}{*{20}{c}}
   {{z_1},{z_2}, \ldots ,{z_r}}  \\
   {k_1,k_2,\ldots,k_r}  \\
\end{array}} ;x\!\right)&=\frac1{\prod_{j=1}^r (k_j-1)!}\int_{0<t_1<t_2<\cdots<t_r<x} \left(\prod_{j=1}^{r-1} z_j(t_j)\Big(\int_{t_j}^{t_{j+1}} z_0(t)dt \Big)^{k_j-1}\right)\nonumber\\
&\qquad\qquad\qquad\qquad\times z_r(t_r)\Big(\int_{t_r}^{x} z_0(t)dt \Big)^{k_r-1}dt_1\cdots dt_r.
\end{align}
Replacing $k_r$ by $k_r+m-1$ and summing over all compositions satisfying $k_1+\cdots+k_r=k+r-1$, we obtain
\begin{align}\label{proof-add-one}
&\sum_{k_1+\cdots+k_r=k+r-1,\atop k_1,\ldots,k_r\geq1} \binom{k_r+m-2}{m-1} L_{z_0}\left(\!\! {\begin{array}{*{20}{c}}
   {{z_1}, \ldots,z_{r-1},{z_r}}  \\
   {k_1,\ldots,k_{r-1},k_r+m-1}  \\
\end{array}} ;x\!\right)\nonumber\\
&= \frac1{(k-1)!(m-1)!}\int_{0<t_1<t_2<\cdots<t_r<x} z_1(t_1)\cdots z_r(t_r) \Big(\int_{t_r}^x z_0(t)dt\Big)^{m-1} \nonumber\\&\qquad\qquad\qquad\qquad\times \left(\int_{t_1}^{t_2} z_0(t)dt+ \cdots +\int_{t_{r-1}}^{t_r} z_0(t)dt+\int_{t_r}^{x} z_0(t)dt\right)^{k-1}dt_1\cdots dt_r\nonumber\\
&=\int_{0<t_1<t_2<\cdots<t_r<x} \frac{\Big(\int_{t_1}^x z_0(t)dt\Big)^{k-1}}{(k-1)!} z_1(t_1)\cdots z_r(t_r) \frac{\Big(\int_{t_r}^x z_0(t)dt\Big)^{m-1}}{(m-1)!}.
\end{align}
It follows readily from the general theory of Chen's iterated integrals \cite[Eqs. (1.6.1-2)]{KTChen1971} (see also \cite[Lemma 4.1]{XuZhao2020b}) that if \(f_i\) (\(i=1,\dots,r\)) are integrable real-valued functions, then the following identity holds:
\begin{align}\label{proof-add-two}
F_{a,b}(f_1,\ldots,f_r)+(-1)^r F_{a,b}(f_r,\ldots,f_1)
=\sum_{j=1}^{r-1}(-1)^{j-1}F_{a,b}(f_j,\ldots,f_1)F_{a,b}(f_{j+1},\ldots,f_r),
\end{align}
where
\[
F_{a,b}(f_1,\ldots,f_r):=\int_{a<t_1<\cdots<t_{r}<b} f_1(t_1)\cdots f_r(t_r)\,dt_1\cdots dt_r.
\]
In view of the identity
\begin{align}\label{proof-add-three}
\int_{0<t_1<t_2<\cdots<t_r<x}z_1(t_1)\cdots z_r(t_r) \frac{\Big(\int_{t_r}^x z_0(t)dt\Big)^{m-1}}{(m-1)!}=L_{z_0}\left(\!\! {\begin{array}{*{20}{c}}
   {{z_{1}}, \ldots,z_{r-1},{z_r}}  \\
   {1,\ldots,1,m}  \\
\end{array}} ;x\!\right),
\end{align}
substituting this into \eqref{proof-add-one} and combining the result with \eqref{proof-add-two}, a direct computation yields the desired formula.
\end{proof}

By specializing the functions \(z_j\) in Theorem~\ref{thm-add-one}, one immediately obtains numerous relations among Arakawa-Kaneko type zeta functions. In particular, this recovers the symmetric sum formulas for the Kaneko-Tsumura multiple \(A\)-functions established in \cite[Theorem 4.4]{XuZhao2020b}. We leave further specializations to the interested reader.

\medskip
\section{Concluding remarks}
By using iterated integrals, we have provided a unified method for proving many results concerning Arakawa-Kaneko type zeta functions, Kaneko-Tsumura type zeta functions and their special values. This work also highlights the historical interplay between algebraic topology, complex analysis, and number theory, as iterated integrals were originally developed by Chen to study path spaces and fundamental groups, and have since found deep applications in the theory of multiple zeta values. These special values have appeared in various other branches of mathematics, including the study of multiple zeta values and their variants, special functions and polynomials, and even $p$-adic theory (see, e.g., \cite{Young2014,PLXZ2024Oct,Umezawa2018}). The connections with $q$-analogues and multiple zeta values are particularly exciting, as these are among the most active areas of research in number theory (see, e.g., \cite{Brown2012,Zhao2007c}). We expect that further applications of these special values will emerge in the future, for instance in the study of modular forms, quantum invariants, and arithmetic geometry, where iterated integrals and their generalizations are likely to play a prominent role.

{\bf Funding}  Both Ce Xu and Jianqiang Zhao thank Professor Chengming Bai for his invitation to visit the Chern Institute of Mathematics and for the support of the institute's Visiting Scholars Program. Ende Pan is supported by the Zhejiang Provincial Department of Education Research Project (Grant No. Y202559171). Ce Xu is supported by the General Program of the Natural Science Foundation of Anhui Province (Grant No. 2508085MA014). Jianqiang Zhao is supported by the Jacobs Prize from The Bishop's School.

{\bf Availability of data and materials} Not applicable.

{\bf Competing interests} The authors declare that they have no competing interests.

\end{document}